\documentclass[11pt]{article}

\usepackage{amsmath,amsthm,amsfonts,amssymb, latexsym, esint} 
\usepackage[usenames]{color}

\newtheorem{thm}{Theorem}[section]

\newtheorem{lem}[thm]{Lemma}
\newtheorem{prop}[thm]{Proposition}
\newtheorem{cor}[thm]{Corollary}

\newtheorem*{remark}{Remark}

\theoremstyle{definition}
\newtheorem*{defin}{Definition}

\DeclareMathOperator{\hess}{Hess}
                                                                                
\DeclareMathOperator{\Div}{div}                                                                                

\newcommand{\ric}{\mathrm{Ric}}

\newcommand\rr{\mathbb{R}}

\newcommand\zz{\mathbb{Z}}

\newcommand\isom{\cong}

\newcommand\too{\longrightarrow}

\DeclareMathOperator{\tr}{tr}

\newcommand\gen{\mathfrak{g}}     
                    
\allowdisplaybreaks                                                                            
\begin{document}

\title{The Penrose inequality for asymptotically locally hyperbolic spaces with nonpositive mass}
\author{Dan A. Lee\and Andr\'e Neves}
\maketitle
\begin{abstract}
 In the  asymptotically {\em locally }hyperbolic setting it is possible to have metrics with scalar curvature $\ge-6$ and \emph{negative} mass when the genus of the conformal boundary at infinity is positive. Using inverse mean curvature flow, we prove a Penrose inequality for these negative mass metrics. The motivation comes from a previous result of P. Chru\'{s}ciel and W. Simon, which  states that the Penrose inequality we prove implies a static uniqueness theorem for negative mass Kottler metrics. 
 \end{abstract}

\section{Introduction}

The Penrose inequality for asymptotically flat $3$-manifolds $M$ with mass $m$ and nonnegative scalar curvature states that if $\partial M$ is an outermost minimal surface (\textit{i.e.}, there are no compact minimal surfaces separating $\partial M$ from infinity), then
$$m\geq \sqrt{\frac{A}{16\pi}},$$
where $A$ is the area of $\partial M$.

G. Huisken and T. Ilmanen \cite{Huisken-Ilmanen:2001} first proved this inequality for $A$ equal to the largest area of a connected component of $\partial M$. H. Bray \cite{Bray:2001} later proved the more general inequality described above using a different method, and this result was later extended to dimensions less than $8$ by Bray and the first author \cite{Bray-Lee:2009}.

We are interested in an analog of this theorem for  a class of asymptotically locally hyperbolic manifolds which we now define.
\begin{defin}\label{asym.hyper.definition}
We say that a $C^i$ Riemannian metric $g$ on a smooth manifold  $M^3$ is \emph{$C^i$ asymptotically locally hyperbolic} if there exists a compact set $K\subset M$ and a constant curvature surface $(\hat{\Sigma},\hat{g})$, called the {\em conformal infinity} of $(M,g)$, such that $M\smallsetminus K$ is diffeomorphic to $(1,\infty)\times \hat{\Sigma}$ with the metric satisfying 
\[ g = (\hat{k} +\rho^2)^{-1} d\rho^2 + \rho^2 \hat{g} + \rho^{-1} h +Q, \] 
where 
\begin{itemize}
\item $\hat{k}$ is the constant curvature of $(\hat{\Sigma},\hat{g})$;
\item $\rho$ is the coordinate on $(1,\infty)$;
\item  $Q$ is a $C^i$ symmetric two-tensor on $M$  so that
\[ |Q|_b + \rho|\bar{\nabla} Q|_b + \cdots+ \rho^i |\bar{\nabla}^i Q|_b = o(\rho^{-3}), \]
where $b$ is the hyperbolic metric $(\hat{k} +\rho^2)^{-1} d\rho^2 + \rho^2 \hat{g}$ and $\bar{\nabla}$ are derivatives taken with respect to $b$. We will use the notation $Q=o_i(\rho^{-3})$ as a convenient abbreviation;
\item $h$ is a $C^i$ symmetric two-tensor on $\hat{\Sigma}$ depending on $\rho$ in such a way that there exists a function $\mu$ on $\hat{\Sigma}$, called the \emph{mass aspect function}, such that
$\lim_{\rho\to\infty}\tfrac{3}{4}\tr_{\hat{g}} h=\mu$ , where the  convergence is in $C^i$.
\end{itemize}
For the sake of convenience, we assume that $\hat{k}$ is $1$, $0$, or $-1$, and in the case $\hat{k}=0$, we further assume that $|\hat{\Sigma}|_{\hat{g}}=4\pi$. (These assumptions simply serve the purpose of normalization.)

Finally, we define the \emph{mass} to be 
\[ m=\fint_{\hat{\Sigma}} \mu\,d\hat{g}=\frac{1}{|\hat\Sigma|_{\hat g}}\int_{\hat\Sigma} \mu\, d\hat{g}, \]
and we also define
\[ \bar{m}=\sup_{\hat{\Sigma}} \mu.\] 
\end{defin}

An important class of asymptotically locally hyperbolic manifolds is given by the Kottler metrics (see  \cite{Chrusciel-Simon:2001} for instance), which are static metrics with cosmological constant $\Lambda=-3$.

\begin{defin}
Let $(\hat{\Sigma}, \hat{g})$ be a surface with constant curvature $\hat{k}$ equal to $1$, $0$, or $-1$, with area equal to $4\pi$ in the $\hat{k}=0$ case.  Let $m\in\rr$ be large enough so that the function 
\[V(r)=\sqrt{r^2+\hat{k}-\frac{2m}{r}}\] 
has a nonnegative zero. Let $r_m$ be the largest zero of $V$,  and define the metric
\[g=V^{-2}dr^2+r^2 \hat{g}\] 
on $(r_m,\infty)\times\hat{\Sigma}$. Define $(M,g)$ to be the metric completion of this Riemannian manifold.
We say that $(M,g)$  is a \emph{Kottler space with conformal infinity $(\hat{\Sigma},\hat{g})$ and mass $m$}.
\end{defin}
\begin{remark}
\begin{itemize}
\item The Kottler metrics have scalar curvature $R=-6$. 
\item The most familiar situation is when $\hat{\Sigma}$ is a sphere $(\hat{k}=1)$, in which case the metric is also called an \emph{anti-de Sitter--Schwarzschild metric} in the literature.
\item  As long as $V(r)$ has a \emph{positive} largest zero $r_m$, we obtain $M=[r_m, \infty)\times \hat \Sigma$ with $\partial M=\{r_m\}\times\hat\Sigma$ as an outermost minimal surface boundary. One can see that these metrics are asymptotically locally hyperbolic by performing a substitution, in which case one has $h= \frac{2}{3}m\hat g$.  
\item In order for $V(r)$ to have a \emph{positive zero}, we must have $m>0$ when $\hat k=0$ or $\hat k=-1$.
However, when $\hat{k}=-1$,  the parameter $m$ need not be positive but only greater than a critical mass $m_\text{crit}=-\frac{1}{3\sqrt{3}}$. 
\item When the largest zero of $V(r)$ is exactly $0$, we say that $(M,g)$ is a \emph{critical} Kottler space.  When $\hat{k}=-1$ and $m=m_\text{crit}$, the metric $g$ on $(0, \infty)\times \hat \Sigma$
is a two-ended complete Riemannian manifold, with one end asymptotically locally hyperbolic and the other end asymptotic to the cylindrical metric $dt^2+\frac 1 3 \hat g$ on $\rr\times \hat \Sigma$. When $\hat{k}=0$ and $m=0$, the metric $g$ can be written as  $dt^2+e^{2t} \hat g$ on $\rr\times \hat \Sigma$ after a coordinate change, and of course, when $\hat{k}=1$ and $m=0$, the completed metric $(M,g)$ is just hyperbolic space.
\end{itemize}
\end{remark}

We can now state our main theorem.

\begin{thm}[Penrose Inequality for nonpositive mass]\label{penrose.thm}
Let $(M^3,g)$ be a   $C^2$ asymptotically locally hyperbolic manifold with $\bar m\leq 0$ and conformal infinity $(\hat{\Sigma}, \hat{g})$, whose genus is $\gen$.

Assume that $R\geq-6$, $\partial M$ is an outermost minimal surface, and  there is a boundary component $\partial_1 M$ of genus $\gen$. 
 Then
\begin{equation}\label{penrose.inequality}
 \bar{m}
\ge \frac{1}{\gamma} \sqrt{\frac{A}{16\pi}}\left(1-\gen+\frac{A}{4\pi}\right),
\end{equation}
where  $A$ is the area of $\partial_1M$,
and  $\gamma=\left(\max\{1, \gen -1\}\right)^{3/2}$ is a topological constant.
 
 Furthermore, equality occurs if and only if $(M,g)$ is isometric to the Kottler space with infinity $(\hat{\Sigma}, \hat{g})$ and mass $\bar{m}$.
\end{thm}

If one replaces $\bar{m}$ by $m$ and removes the  $\bar m\le0$ condition in our theorem,\footnote{Although we are unable to replace $\bar{m}$ by $m$ in general, we do obtain a slightly stronger inequality than (\ref{penrose.inequality}). See Lemma \ref{2/3-power}.}  
one obtains the natural analog of Huisken and Ilmanen's Penrose inequality \cite{Huisken-Ilmanen:2001} in the asymptotically locally hyperbolic setting. Versions of this statement have been conjectured by P. Chru\'{s}ciel and W. Simon \cite[Section VI]{Chrusciel-Simon:2001}, and in the $\gen=0$ case, by X. Wang \cite[Section 1]{Wang:2001}. The graph case of the conjecture has recently been established by L. de Lima and F. Gir\~{a}o  when $m\ge0$ \cite{deLima-Girao:2013}. See also a related volume comparison result by S. Brendle and O. Chodosh~\cite{Brendle-Chodosh:2013}.

\begin{cor}\label{BoundaryPMT}
Let $(M^3,g)$ be a   $C^2$ asymptotically locally hyperbolic manifold with conformal infinity of genus $\gen$.

Assume that $R\geq-6$, $\partial M$ is an outermost minimal surface, and  there is a boundary component $\partial_1 M$ of genus $\gen$.  Then
\begin{align*}
& \bar{m}>0 \quad\text{for $\gen=0$ or $1$,}\\
& \bar{m} > -\frac{1}{3\sqrt{3}}\quad \text{for $\gen>1$}.
\end{align*}
\end{cor}
\begin{proof}
	If $\gen=0$ or $\gen=1$ and $\bar m\leq 0$ we obtain at once from Theorem \ref{penrose.thm} that $|\partial_1 M|=0$, which is a contradiction.
	
	If $\gen>1$, then it is an easily verifiable fact that
	\[\sqrt{\frac{x}{16\pi}}\left(1-\gen+\frac{x}{4\pi}\right)\geq -\frac{(\gen-1)^{3/2}}{3\sqrt{3}}\quad\mbox{for all }x\geq 0.\]
	This inequality and Theorem \ref{penrose.thm} imply at once that $\bar m\geq -\frac{1}{3\sqrt{3}}$. If equality holds then we are in the equality case of Theorem \ref{penrose.thm} and so $M$ must be isometric to a critical Kottler space of mass $ -\frac{1}{3\sqrt{3}}$. But this is impossible because a critical Kottler space has no compact minimal surfaces. (They are foliated by strictly mean convex surfaces.)
\end{proof}
Of course, the $\gen=0$ case of Corollary \ref{BoundaryPMT} is just the positive mass theorem for asymptotically hyperbolic manifolds, proved by Chru\'{s}ciel and M. Herzlich \cite{Chrusciel-Herzlich:2003} and by Wang \cite{Wang:2001}, for the case of a manifold with minimal boundary, proved by V. Bonini and J. Qing \cite{Bonini-Qing:2008}.  
While the assumption on the boundary is not desirable, we note that  there exist examples
 (the AdS solitons due to G. Horowitz and R. Myers \cite{Horowitz-Myers:1999})
 of asymptotically locally hyperbolic manifolds  with $\gen=1$, no boundary, and \emph{negative} mass.

We note that the inverse mean curvature flow technique allows us to a give a new proof of a weakened version of the positive mass theorem for asymptotically hyperbolic manifolds that was mentioned above.
\begin{thm}[Positive $\bar{m}$ theorem for asymptotically hyperbolic manifolds]\label{AHPMT}
Let $(M^3,g)$ be a complete $C^2$ asymptotically hyperbolic\footnote{Here we define this to mean asymptotically locally hyperbolic with infinity equal to the round sphere.} manifold (with or without a minimal boundary) with scalar curvature $R\geq-6$.  Then $\bar{m}\ge0$. Moreover, if $\bar{m}=0$, then $(M,g)$ must be hyperbolic space.
\end{thm}

We prove Theorem \ref{penrose.thm} following the inverse mean curvature flow theory developed by Huisken and Ilmanen in \cite{Huisken-Ilmanen:2001}. The general idea is to flow $\partial_1 M$ outward with speed inversely proportional to the mean curvature and obtain a (weak) flow of surfaces $(\Sigma_t)_{t\geq 0}$ (where $\Sigma_0=\partial_1 M$). To each compact surface one considers its \emph{Hawking mass} to be
\begin{equation}\label{HawkingMass}
m_H(\Sigma):=\sqrt{\frac{|\Sigma|}{16\pi}}\left(1-\gen-\frac{1}{16\pi}\int_{\Sigma}(H^2-4)\right),
\end{equation}
where $|\Sigma|$ denotes the area of $\Sigma$, and $H$ is its mean curvature. Observe that our notation $m_H(\Sigma)$ leaves out the dependence on $\gen$. The key property of inverse mean curvature flow is that $m_H(\Sigma_t)$ is non-decreasing in time. Therefore, since $m_H(\Sigma_0)$  coincides with the right-hand side of the inequality in Theorem \ref{penrose.thm}, the desired result follows if 
\[\lim_{t\to\infty}m_{H}(\Sigma_t)\leq \bar{m}\cdot\left(\max\{1, \gen -1\}\right)^{3/2}.\]

In the asymptotically locally hyperbolic setting this inequality is subtle, and in fact the second author constructed  well-behaved examples (with $\gen=0$) where the above inequality does not hold \cite{Neves:2010}. The central observation of this paper is that this inequality holds if  $\bar m\leq 0$. The difference between our result here and the one in \cite{Neves:2010} can be traced to the fact that if the mass aspect is positive, then a desired inequality goes in the wrong direction, whereas for nonpositive mass aspect, it goes in the right direction. See the end of the proof of Lemma \ref{2/3-power} to see the exact place where the condition  $\bar m\leq 0$ is used.

\vskip 0.02in
\noindent{\em Acknowledgements:} The authors would like to thank Piotr Chru\'{s}ciel  and Walter Simon for bringing this problem to our attention and for their interest in this work. We also thank Richard Schoen for some helpful conversations.

\section{An application to static uniqueness}\label{StaticUniqueness}

Chru\'{s}ciel and Simon proved in  \cite{Chrusciel-Simon:2001} that Theorem \ref{penrose.thm} implies a uniqueness theorem for static metrics with negative mass that we now explain.

\begin{defin}
We say that $(M^3,g,V)$ is a \emph{complete vacuum static data set with  cosmological constant $\Lambda=-3$} if and only if $M$ is a smooth manifold (possibly with boundary) equipped with a complete $C^2$ Riemannian metric $g$ and a nonnegative $C^2$ function $V$ such that $\partial M=\{V=0\}$ and
\begin{gather}
\Delta_g V = -\Lambda V \label{static1}\\
\ric(g)=\frac{1}{V}\hess_g V+\Lambda  g.\label{static2}
\end{gather}
\end{defin}

If $(M^3,g,V)$ is a vacuum static data set as above, then the Lorentzian metric $h=-V^2dt^2+g$ on $\rr\times M$ is a solution to Einstein's equations with cosmological constant $\Lambda$. In the $\Lambda=0$ case a 1987 result of  G. Bunting and A. Masood-ul-Alam  \cite{Bunting-Masood:1987} shows that  Schwarzschild spaces are the only asymptotically flat vacuum static data sets (with the case of connected boundary originally proved by H. M\"{u}ller zum Hagen, D. Robinson, and H. Seifert in 1973 \cite{Muller-Robinson-Seifert:1973}).

We are interested in a similar characterization of the Kottler metrics defined in the previous section, which are known to be vacuum static data sets with cosmological constant  $\Lambda=-3$.  A static uniqueness theorem for the hyperbolic space was proved in work of Boucher-Gibbons-Horowitz 
\cite{Boucher-Gibbons-Horowitz:1984}, Qing \cite{Qing:2003}, and Wang \cite{Wang:2005}.

The following definition is equivalent to the one given in \cite[Section III.A]{Chrusciel-Simon:2001}.
\begin{defin}
Let $i\ge 2$. We say that a complete vacuum static data set $(M^3,g,V)$ with cosmological constant $\Lambda=-3$ is \emph{$C^i$ conformally compactifiable} if $g$ and $V$ are $C^i$ and there exists a smooth compact manifold $M'$ with boundary and a $C^{i+1}$ embedding of $M$ into $M'$ such that $M'\isom M\cup\partial_\infty M$,\\
\indent $\bullet$ the function $V^{-1}$ extends to a $C^i$ function on $M'$ with $d(V^{-1})\ne0$ at $\partial_\infty M$, and\\
\indent $\bullet$ the formula $\hat{g}=V^{-2}g$ near $\partial_\infty M$ defines a Riemannian metric.
\end{defin}

If $(M,g,V)$ is a complete vacuum static data set, the fact that \hbox{$\hess_g V=0$} on $\partial M$ implies that $|\nabla V|$ is constant on each component of $\partial M$, and we call this constant the \emph{surface gravity} $\kappa$ of that component.   

The surface gravity is strictly positive for the following reason: Given $p\in \partial M$, let $\gamma$ be the unit speed geodesic that starts at $p$ perpendicular to  $\partial M$ and set $f(t)=V(\gamma(t))$. From the static equation we see the existence of $c>0$ and $t_0>0$ so that
$f''\leq cf$ for all $0\leq t\leq t_0.$ Standard o.d.e.\ comparison shows that if $f'(0)=f(0)=0$ then $f(t)\leq 0$ for all $0\leq t\leq t_0,$ which is impossible because $V$ is strictly positive on the interior of $M$.

The (non-critical) Kottler metrics have exactly one component of $\partial M$ and, for fixed $\hat{k}$, there is a bijection between possible surface gravities in $(0,\infty)$ and possible masses in $(-\frac{1}{3\sqrt{3}},\infty)$ when $\hat{k}=-1$, or in $(0,\infty)$ when $\hat{k}$ is $0$ or $1$. Therefore, for fixed $\hat{k}$ we can define a bijection $m(\kappa)$ according to the fixed relationship between mass and surface gravity for Kottler metrics whose infinities have curvature equal to $\hat{k}$.  For this bijection, when $\hat k=-1$ one has  $\lim_{\kappa\to 0} m(\kappa)=-\frac{1}{3\sqrt{3}}$ and $m(1)=0$  (see \cite[Section II]{Chrusciel-Simon:2001} for details).

We can now state a static uniqueness theorem for Kottler metrics of negative mass, which will follow from Theorem \ref{penrose.thm} combined  with some of the results of \cite{Chrusciel-Simon:2001} that we will describe later.

\begin{thm}[Static uniqueness with nonpositive mass]\label{StaticUniquenessThm}
Let $(M^3,g,V)$ be a complete vacuum static data set with cosmological constant $\Lambda=-3$, and assume that it is $C^5$ conformally compactifiable with conformal infinity $\partial_\infty M$ of constant curvature $\hat{k}=-1$.

Assume that there is a component $\partial_1 M$ of $\partial M$ such that $\partial_1 M$ is homeomorphic to $\partial_\infty M$ and $\partial_1 M$ has the largest surface gravity $\kappa$ of any component.

If $m(\kappa) \le 0$ (or equivalently $\kappa\leq 1$), then $(M, g)$ must be isometric to the Kottler metric with infinity $\partial_\infty M$ and mass $m(\kappa)$, while $V$ is equal to the usual static potential of the Kottler metric, up to a constant multiple.
\end{thm}
Before we present the proof some comments are in order.
\begin{itemize}

\item It follows from Lemma \ref{outermost} and the proof of Lemma \ref{topology} that $\partial_1 M$ can never have larger genus than $\partial_\infty M$. It would be nice to also rule out the possibility that $\partial_1 M$ has strictly smaller genus.

\item The Horowitz-Myers AdS solitons  \cite{Horowitz-Myers:1999} described in the Introduction do not only have
 have negative mass, no boundary, and $\hat k=0$, but they are also \emph{static}. A static uniqueness theorem for these examples was proved by G. Galloway, S. Surya, and E. Woolgar \cite{Galloway-Surya-Woolgar:2003}. 

\item It would be interesting to remove the condition on $m(\kappa)$.

\end{itemize}

\subsection{Proof of Theorem \ref{StaticUniquenessThm}}

The following proposition is a consequence of Theorem I.1 and Proposition III.7 of \cite{Chrusciel-Simon:2001}, together with a coordinate change.

\begin{prop}\label{StaticAsymptotics}
Let $i\ge 3$. Let $(M^3,g,V)$ be a $C^i$ conformally compactifiable complete vacuum static data set with cosmological constant $\Lambda=-3$. Further assume that the induced metric $\hat{g}$ on $\partial_\infty M$ (as defined in the previous section) has locally constant Gauss curvature $\hat{k}$ equal to  $1$, $0$, or $-1$. 

Then 
$\partial_\infty M$ is connected, $(M,g)$ is asymptotically locally hyperbolic (as defined in the Introduction) with conformal infinity $(\partial_\infty M, \hat{g})$, and  
\[V^2=\rho^2+\hat{k} -\frac{4\mu}{\rho}+o_1(\rho^{-1}),\]
where $\rho$ is the coordinate used in the definition of asymptotically locally hyperbolic, and $\mu$ is the mass aspect.
\end{prop}
In particular, static data sets with $\Lambda=-3$ have a well-defined mass $m$ and $\bar m$.

The key theorem of \cite{Chrusciel-Simon:2001} for the purposes of this article is Theorem I.5:
\begin{thm}[Chru\'{s}ciel-Simon]\label{CSTheorem}
Let $(M^3,g,V)$ be a $C^3$ conformally compactifiable complete vacuum static data set with cosmological constant $\Lambda=-3$, conformal infinity $(\partial_\infty M,\hat g)$ of constant curvature $\hat{k}=-1$, and $\partial M\ne\emptyset$.

Let $\partial_1 M$ denote the boundary component with the largest surface gravity~$\kappa$ and suppose $m_0:=m(\kappa)\le 0$ (i.e., $0<\kappa\leq 1$). 

If $(M_0, g_0, V_0)$ denotes the Kottler space with infinity $(\partial_\infty M, \hat{g})$ and mass $m_0$ (\textit{i.e.}, the one with surface gravity $\kappa$), then
\[ \frac{\chi(\partial_1 M)}{|\partial_1 M|} \ge \frac{\chi(\partial M_0)}{|\partial M_0|_{g_0} }\quad\mbox{and}\quad \bar{m} \le m_0,\]
where $|\partial M_0|_{g_0}$ is the area  with respect to $g_0$.
\end{thm}
In the case where $\partial_1 M$ has the same genus as $\partial_\infty M$, this theorem provides a simple comparison between the masses and boundary areas of a vacuum static data set and its so-called \emph{reference solution} $(M_0, g_0, V_0)$. For the sake of completeness we provide the proof of this theorem in Section \ref{AreaMassInequalities}.

\begin{lem}\label{outermost}
Let $(M^3,g,V)$ be a $C^3$ asymptotically locally hyperbolic, complete vacuum static data set with cosmological constant $\Lambda=-3$. If $\partial M\ne\emptyset$, then $\partial M$ is an outermost minimal surface. In fact, there are no compact minimal surfaces in the interior of $M$.
\end{lem}
\begin{proof}
First note that the static equations imply that $\partial M$ is totally geodesic, so we need only show that there are no other compact minimal surfaces. Consider the trapped region $K$ of $M$, which is the union of all compact minimal surfaces in $M$, together with all regions of $M$ that are bounded by these minimal surfaces. The boundary of the trapped region, $\partial K$, must itself be a smooth compact minimal surface. (See the proof of Lemma 4.1(i) of \cite{Huisken-Ilmanen:2001}.) Following \cite{Huisken-Ilmanen:2001}, we define the \emph{exterior region} $M'$ of $M$ to be the metric completion of $M\smallsetminus K$. Thus $(M',g, V)$ is a vacuum static data set, except for the requirement that $\{x\in M'\,|\,V(x)=0\}=\partial M'$. The exterior region $M'$ has a strictly outward minimizing minimal boundary and no interior compact minimal surfaces (see \cite{Huisken-Ilmanen:2001}), where \emph{strictly outward minimizing} means that $|\partial M'|$ is strictly less than the area of any other surface that encloses it. 

Suppose that $M$ has a compact minimal surface other than $\partial M$. Since $V>0$ away from $\partial M$, it follows from the definition of $M'$ that $V$ does not vanish identically on $\partial M'$.
We consider the  outward normal flow of surfaces $(\Sigma_t)_{t\ge0}$  with initial condition $\Sigma_0=\partial M'$ that flows with speed $V$. Since $V\ge0$ does not vanish on $\partial M'$, this flow is nontrivial.
According to the formula for the variation of mean curvature, we see that the mean curvature of $\Sigma_t$ evolves according to 
\begin{align*}
\frac{\partial H}{\partial t} &= -\Delta_{\Sigma_t} V -(\ric(\nu,\nu)+|A_{\Sigma_t}|^2)V \\
&= -(\Delta_{g} V - \nabla_\nu\nabla_\nu V + \langle H, \nabla V\rangle) -(\nabla_\nu\nabla_\nu V-3V +|A_{\Sigma_t}|^2 V) \\
&=  -\langle H, \nabla V\rangle -|A_{\Sigma_t}|^2 V, 
\end{align*}
where $\nu$ is the outward unit normal, and $A_{\Sigma_t}$ is the second fundamental form. Since $V\ge 0$, it follows that $\Sigma_t$ must have $H\le0$ for all small $t$.
By the first variation of area formula, $\Sigma_t$ must have area less than or equal to that of $\Sigma_0=\partial M'$. But this contradicts the strictly outward minimizing property of $\partial M'$. 
\end{proof}

We can now prove Theorem \ref{StaticUniquenessThm} following the description in \cite{Chrusciel-Simon:2001}.

Let $(M, g, V)$ be as in the statement of the theorem. In particular, all of the hypotheses of Chru\'{s}ciel-Simon's Theorem (Theorem \ref{CSTheorem}) are satisfied and so $\bar m\leq 0$. Hence  Proposition \ref{StaticAsymptotics} and Lemma \ref{outermost} imply that all of the hypotheses of our Penrose inequality (Theorem \ref{penrose.thm}) are also satisfied.

Since we are assuming that $\partial_1 M$ is homeomorphic to $\partial_\infty M$, Theorem \ref{CSTheorem}  tells us that
\[ A \ge A_0\quad\mbox{and}\quad\bar{m}\le m_0,\]
where $A$ is the area of $\partial_1M$ and $A_0$ is the area of $\partial M_0$ in the reference solution. Moreover, since the curvature $\hat{k}$ of $\hat{g}$ is equal to $-1$ our Penrose inequality (Theorem \ref{penrose.thm}) tells us that 
\begin{align}
\bar{m}(\gen -1)^{3/2}
& \ge \sqrt{\frac{A}{16\pi}}\left(1-\gen+\frac{A}{4\pi}\right).\label{mbar}
\end{align}
It is convenient to define constants 
$$
\mathfrak{r}:=\sqrt{\frac{A}{4\pi(\gen-1)}}\quad\mbox{and}\quad
r_0:=\sqrt{\frac{A_0}{4\pi(\gen-1)}},
$$
so that we have $\mathfrak{r}\ge r_0.$

Inequality (\ref{mbar}) then becomes
$$
\bar{m} \ge \frac{1}{2}\mathfrak{r}(- 1 + \mathfrak{r}^2) \implies
2\bar{m} + \mathfrak{r}-\mathfrak{r}^3\ge0.
$$
Meanwhile, on the reference space we know that $r_0$ is the largest root of 
\[2m_0 +r_0-r_0^3=0. \]
and so $r_0\ge \frac{1}{\sqrt{3}}$ using elementary reasoning. Thus
\begin{align*}
0\le 2\bar{m} + \mathfrak{r}-\mathfrak{r}^3
& =(2\bar{m} + \mathfrak{r}-\mathfrak{r}^3) - (2m_0 +r_0-r_0^3)\\
& = 2(\bar{m}-m_0) + (\mathfrak{r}-r_0) - (\mathfrak{r}^3-r_0^3)\\
&= 2(\bar{m}-m_0) + (\mathfrak{r}-r_0)[1 - (\mathfrak{r}^2+\mathfrak{r}r_0+r_0^2)]\\
& \le 2(\bar{m}-m_0) + (\mathfrak{r}-r_0)(1 - 3r_0^2)
\le 0.
\end{align*}
Therefore all of the inequalities must be equalities and so it follows from the rigidity part of Theorem \ref{penrose.thm} that $(M,g)$ is isometric to the Kottler space with infinity $\partial_\infty M$ and mass $m_0$.

It is simple to check that any two static potentials on the Kottler space $(M,g)$ must be proportional and so $V$ is the usual static potential, up to a constant multiple.

\section{Proof of Theorem \ref{penrose.thm}}\label{IMCF}

Assume that $(M^3,g)$ be a   $C^2$ asymptotically locally hyperbolic manifold with $R\geq-6$ such that $\partial M$ is an outermost minimal surface. In order to establish existence of the weak inverse mean curvature flow, we first need to find a weak subsolution.

\begin{lem}\label{weak.subsolution} Let $(M^3,g)$ be a $C^2$ asymptotically locally hyperbolic metric with radial coordinate $\rho$ as in the definition of asymptotically locally hyperbolic.
There exists $r_0$ so that for all $r\geq r_0$,
$$\Sigma^-_t=\{\rho=(r+1)e^{t/2}-1\}\quad\mbox{and}\quad \Sigma^+_t=\{\rho=(r-1)e^{t/2}+1\}, \quad t\geq 0 $$
are, respectively, subsolutions and supersolutions for inverse mean curvature flow with initial condition $\{\rho=r\}.$
\end{lem}
\begin{proof}
We will prove that  $\Sigma^+_t$ is a supersolution. (The proof for $\Sigma^-_t$ is similar.) Using the asymptotics of $g$, one can see that the inverse mean curvature of the constant $\rho$ sphere in $(M,g)$ is
\[ H^{-1} = \frac{\rho}{2\sqrt{\hat{k}+\rho^2}} + O(\rho^{-2}).\]
Since that $\Sigma^+_t$ is just the constant $\rho$ sphere with $\rho=(r-1)e^{t/2}+1$, we see that the speed of the flow is just 
\[ \frac{d\rho}{dt} |\partial_\rho|_g = \tfrac{1}{2}(r-1)e^{t/2}[(\hat{k}+\rho^2)^{-1/2} + O(\rho^{-3})] = \frac{\rho}{2\sqrt{\hat{k}+\rho^2}} -  \frac{1}{2\sqrt{\hat{k}+\rho^2}}+ O(\rho^{-2}).\]
Clearly, for sufficiently large $\rho$, this speed is less than $H^{-1}_{\Sigma^+_t}$, showing that $\Sigma^+_t$ is a supersolution.
\end{proof}

Given Lemma \ref{weak.subsolution}, we may now apply Huisken and Ilmanen's Weak Existence Theorem 3.1 of \cite{Huisken-Ilmanen:2001} to find a weak solution $u$ for inverse mean curvature flow with initial condition $\partial M$.  More precisely, $u$ is a proper, locally Lipschitz nonnegative function $u$ defined on $M$ with $u=0$ on $\partial M$ that satisfies a certain variational property (defined on page 365 of \cite{Huisken-Ilmanen:2001}). The surfaces $\Sigma_t:=\partial\{u<t\}$ are $C^{1,\alpha}$ and strictly outward minimizing,\footnote{Huisken and Ilmanen instead describe the region enclosed by $\Sigma_t$ as a strictly minimizing hull \cite{Huisken-Ilmanen:2001}.} as defined in the proof of Lemma \ref{outermost}.
In particular, each $\Sigma_t$ is mean convex. There are only countably many ``jump times,'' that is, values of $t$ for which $\Sigma_t:=\partial\{u<t\}$ does not equal $\Sigma_t^+:=\partial(\text{int}\{u\le t\})$.
In a nonrigorous sense, $\Sigma_t$ may be regarded as flowing by smooth inverse mean curvature flow, except when it ceases to be strictly outward minimizing, at which time it ``jumps'' to a strictly outward minimizing surface $\Sigma_t^+$ of equal area.

In case $\partial M$ is not connected, Huisken and Ilmanen explained how one can single out a component $\partial_1 M$ of $\partial M$ as the initial surface while treating the other components of $\partial M$ as ``obstacles.'' See Section 6 of \cite{Huisken-Ilmanen:2001} for details. Essentially, we arbitrarily ``fill in'' all other components $\partial_2 M,\ldots,\partial_n M$ of $\partial M$ to obtain a new space $\tilde M$ and then run the weak inverse mean curvature flow in $\tilde{M}$ with initial condition $\partial_1 M$, except that whenever the surface $\Sigma_t$ is about to enter the filled-in region, we jump to a connected strictly outward minimizing surface $F$ enclosing both $\Sigma_t$ and one or more of the filled-in regions. 
 We then restart the flow with initial condition $F$.

There is another important alteration introduced by Huisken and Ilmanen \cite[Section 4]{Huisken-Ilmanen:2001}. We consider the \emph{exterior region} $M'$ of $M$, as defined in the proof of Lemma \ref{outermost}.
Since $\partial M$ was the outermost minimal surface of $M$, it follows that $\partial M$ is still part 
 of the boundary of $M'$, but now there might be more minimal boundary components. The exterior region $M'$ is an improvement over $M$ because it is completely free of compact minimal surfaces in its interior. We will actually run the weak inverse mean curvature flow in the exterior region of $M$ rather than in $M$ itself. So for our proof of Theorem \ref{penrose.thm}, we may assume without loss of generality that $M$ is an exterior region.

\subsection{Monotonicity of inverse mean curvature flow}
Fix an integer $\gen$ and recall our definition of the Hawking mass of a surface~$\Sigma$ in $(M,g)$ to be
\begin{equation*}
m_H(\Sigma):=\sqrt{\frac{|\Sigma|}{16\pi}}\left(1-\gen-\frac{1}{16\pi}\int_{\Sigma}(H^2-4)\right).
\end{equation*}

The proof of  the Geroch Monotonicity Formula 5.8 in \cite{Huisken-Ilmanen:2001} adapts straightforwardly to the locally hyperbolic setting to show the following:
\begin{thm}[Huisken-Ilmanen]\label{MassDifference1}
Let $(M^3,g)$ be a complete, one-ended, $C^2$ asymptotically locally hyperbolic manifold with outermost minimal boundary, and let $\partial_1 M$ be one of its boundary components. Let $\Sigma_t$ be a weak solution to inverse mean curvature flow (possibly with obstacles, as described above), with initial surface $\partial_1 M$. Then for $0\le \xi<\eta$, if there are no obstacles between $\Sigma_\xi$ and $\Sigma_\eta$, then
\begin{multline} \label{MassDifference}
m_H(\Sigma_\eta)-m_H(\Sigma_\xi)
 \ge \frac{1}{2}(16\pi)^{-3/2} \int_\xi^\eta |\Sigma_t|^{1/2}\Bigg[ 8\pi(\hat{\chi}-\chi(\Sigma_t))\\
 + \int_{\Sigma_t} \left(2(R+6) +|\AA|^2+4H^{-2}|\nabla H|^2\right)\Bigg]\,dt,
\end{multline}
where  $\hat{\chi}:=2-2\gen$, and $\AA$ is the trace-free part of the second fundamental form.
\end{thm}

To make use of this theorem we need the following lemma.

\begin{lem}\label{topology}
Let $(M^3,g)$ be a complete, one-ended, $C^2$ asymptotically locally hyperbolic manifold, which is an exterior region. 
 Let $\partial_1 M$ be a component of $\partial M$ with genus $\gen$, and let $\Sigma_t$ be a weak solution to inverse mean curvature flow (possibly with obstacles), with initial surface $\partial_1 M$. For all $t$, the surface $\Sigma_t$ is connected and has genus at least $\gen$. In particular, $\hat{\chi}\ge\chi(\Sigma_t)$.
\end{lem}
\begin{proof}
Let $\tilde M$ be the manifold $M$ with the obstacles filled in.
Let $u$ be the function defining the weak flow (possibly with obstacles). For each $t>0$, let $\tilde{M}_t$ be the closure of $\{x\in \tilde{M}\,|\, u(x)< t\}$, so that $\partial{\tilde{M}}_t=\Sigma_t \cup \partial_1 M$. We claim that $\tilde{M}_t$ is connected. If it were not connected, one of the components $\Omega$ of $\tilde{M}_t$ would be disjoint from $\partial_1 M$.  By the variational property that characterizes $u$, one can deduce that $u$ must be constant over $\Omega$. (See the proof of \cite[Connectedness Lemma 4.2(i)]{Huisken-Ilmanen:2001}.) Since $\tilde{M}$ is connected, $\Omega$ must meet $\Sigma_t$, and thus $u=t$ on $\Omega$, which is a contradiction to the definition of~$\tilde{M}_t$. Since $\tilde{M}_t$ is connected, it follows that $M_t:=\tilde{M}_t\cap M$ is connected. Note that $\partial M_t = \Sigma_t \cup \partial_1 M \cup \cdots\cup \partial_k M$, where $\partial_2 M,\ldots,\partial_k M$ is some labeling of the other components of $\partial M$ that touch $M_t$.

The rest of the proof does not use inverse mean curvature flow. It is essentially a topological argument that relies only on the following facts about $\Sigma_t$: There exists a connected manifold $M_t$ whose boundary is $\Sigma_t \cup \partial_1 M \cup\cdots \cup \partial_k M$, $\Sigma_t$ is mean convex, each $\partial_i M$ is minimal, and the $\partial_i M$'s are the only compact minimal surfaces in $M_t$. This last part is where we use the assumption that $M$ is an exterior region.

Let $\Sigma^{1}, \ldots, \Sigma^{\ell}$ be the connected components of $\Sigma_t$. We  minimize area in the isotopy class of $\Sigma^{1}$ in $M_t$. 
Note that Theorem 1 of Meeks, Simon, and Yau applies because $M_t$ has mean convex boundary, as explained in Section~6 of \cite{Meeks-Simon-Yau:1982}. According to Theorem 1 and Remark 3.27 of \cite{Meeks-Simon-Yau:1982}, there exists some surface $\tilde{\Sigma}^{1}$ obtained from $\Sigma^{1}$ via isotopy and a series of $\gamma$-reductions such that each component of $\tilde{\Sigma}^{1}$ is a parallel surface of a connected minimal surface, except for one component that may be taken to have arbitrarily small area. 
Recall that a $\gamma$-reduction is a surgery procedure that deletes an annulus and replaces it with two disks in such a way that the annulus and two disks bound a ball in $M_t$. (See \cite[Section 3]{Meeks-Simon-Yau:1982} for details.) 

Note that $\gamma$-reduction preserves homology class.
 Since the only compact minimal surfaces in $M_t$ are the $\partial_i M$'s, and because a surface of small enough area must be homologically trivial, it follows that
\begin{equation}\label{homologous}
  [\Sigma^{1}] = [\tilde{\Sigma}^{1}]=
  \sum_{i=1}^{k}n_i [\partial_i M] \text{ in }H_2(M_t,\zz), 
  \end{equation}
for some integers $n_i$. Using the long exact sequence for the pair $(M_t,\partial M_t)$, we have exactness of 
\[ H_3(M_t,\partial M_t)\overset{\partial}{\too} H_2(\partial M_t,\zz)\overset{\iota_*}{\too} H_2(M_t,\zz).\]
Since $M_t$ is connected, $\ker i_*$ must be generated by 
\[\partial [M_t]=\sum_{i=1}^{\ell} [\Sigma^{i} ] -\sum_{i=1}^{k}[\partial_i M],\]
where $\Sigma^{i}$ and $\partial M$ are oriented using the outward normal in $M$ as usual.
Since equation (\ref{homologous}) says that $[\Sigma^{1}] -
   \sum_{i=1}^{k}n_i [\partial_i M]\in \ker i_*$, it follows that $\Sigma_t$ must be connected and hence equal to $\Sigma^{1}$. In particular,
 \[ [\Sigma^{1}] =\sum_{i=1}^{k}[\partial_i M] \text{ in }H^2(M_t,\zz).\]
Since each component of $\tilde{\Sigma}^{1}$ is either isotopic to one of the $\partial_i M$'s (with some orientation) or is null homologous, and since
there are no relations among $[\partial_i M]$ in $H^2(M_t,\zz)$, the previous equation implies that at least one component of $\tilde{\Sigma}^{1}$ is isotopic to $\partial_1 M$. Finally, since $\gamma$-reduction can only reduce the total genus of all components of a surface, we know that $\Sigma_t$ has genus at least as large as that of $\partial_1 M$.
\end{proof}
Note that if we apply the reasoning in the proof of Lemma \ref{topology} above to the ``conformal infinity'' $\hat{\Sigma}$ of $M$, we see that the genus of $\hat{\Sigma}$ is at least as large as the genus of $\partial_1 M$. 
\begin{cor}[Geroch monotonicity]\label{Geroch}
Let $(M^3,g)$ be a complete, one-ended, $C^2$ asymptotically locally hyperbolic manifold, which is an exterior region. 
 Let $\partial_1 M$ be a component of $\partial M$ with genus $\gen$, and let $\Sigma_t$ be a weak solution to inverse mean curvature flow (possibly with obstacles), with initial surface $\partial_1 M$.
Then the Hawking mass of $\Sigma_t$ is nondecreasing in $t$.
\end{cor}
\begin{proof}
The result follows immediately from Theorem \ref{MassDifference1} and Lemma \ref{topology} in the absence of obstacles. Since there are only finitely many obstacles, all that is left to show is that the mass cannot drop when we jump over an obstacle. Let $t$ be the first time that we have to jump over an obstacle, and let $F$ denote the strictly outward minimizing surface that $\Sigma_t$ jumps to. Then we know from \cite[Equation 6.1]{Huisken-Ilmanen:2001} that 
\[ |\Sigma_t|\le |F|\quad\mbox{and}\quad\int_{\Sigma_t} H_{\Sigma_t}^2 \ge \int_F H_F^2,\]
where the second inequality essentially follows from the fact that the strictly minimizing hull of a surface should be minimal away from where it agrees with the original surface. In the case $\gen<2$, these inequalities combine with nonnegativity of $m_H(\Sigma_t)$ (from monotonicity in the absence of obstacles) to immediately show that $m_H(\Sigma_t)\le m_H(F)$, just as in the asymptotically flat case \cite[Section 6]{Huisken-Ilmanen:2001}. 

To handle the  case  $\gen\ge 2$ we proceed as follows.  
\begin{align*}
& m_H(F)-m_H(\Sigma_t)
 =|F|^{1/2} \frac{ m_H(F)}{|F|^{1/2}} - |\Sigma_t|^{1/2} \frac{ m_H(\Sigma_t)}{|\Sigma_t|^{1/2}}\\
& = (|F|^{1/2}- |\Sigma_t|^{1/2}) \frac{ m_H(\Sigma_t)}{|\Sigma_t|^{1/2}} + |F|^{1/2} \left(\frac{ m_H(F)}{|F|^{1/2}}- \frac{ m_H(\Sigma_t)}{|\Sigma_t|^{1/2}}\right) \\
 & =(|F|^{1/2}- |\Sigma_t|^{1/2}) \frac{ m_H(\Sigma_t)}{|\Sigma_t|^{1/2}}  \\
  &\quad+ |F|^{1/2} (16\pi)^{-3/2}\left( -\int_{F}(H_F^2-4) + \int_{\Sigma_t}(H_{\Sigma_t}^2-4)\right)\\
  &\ge(|F|^{1/2}- |\Sigma_t|^{1/2}) \frac{ m_H(\Sigma_t)}{|\Sigma_t|^{1/2}}  
+ |F|^{1/2} (16\pi)^{-3/2}(4 |F| - 4|\Sigma_t|)\\
&\ge(|F|^{1/2}- |\Sigma_t|^{1/2}) \frac{ m_H(\Sigma_t)}{|\Sigma_t|^{1/2}}  
+ |\Sigma_t|^{1/2} (16\pi)^{-3/2}(4 |F| - 4|\Sigma_t|)\\
&=(|F|^{1/2}- |\Sigma_t|^{1/2})\left( \frac{ m_H(\Sigma_t)}{|\Sigma_t|^{1/2}}+   4(16\pi)^{-3/2} |\Sigma_t|^{1/2} ( |F|^{1/2} + |\Sigma_t|^{1/2})\right)\\
&\ge(|F|^{1/2}- |\Sigma_t|^{1/2})\left( \frac{ m_H(\Sigma_t)}{|\Sigma_t|^{1/2}}+   8(16\pi)^{-3/2} |\Sigma_t|\right).
\end{align*}
Therefore it only remains to show that $m_H(\Sigma_t) \ge -8(16\pi)^{-3/2} |\Sigma_t|^{3/2}$. Note that for any $A\geq 0$,
 \[\sqrt{\frac{A}{16\pi}}\left(1-\gen +\frac{1}{4\pi}A\right)\ge-\left(\frac{\gen-1}{3}\right)^{3/2}.\] 
 Taking $A=|\partial_1 M|$ and using monotonicity in the absence of obstacles, we then have
\[ m_H(\Sigma_t) \ge m_H(\partial_1 M)
\ge -\left(\frac{\gen-1}{3}\right)^{3/2}.\]
On the other hand, observe that a stable compact minimal surface of genus~$\gen$ in a 3-manifold with $R\ge-6$ must have area at least $\frac{4\pi}{3}(\gen-1)$. (This follows from a standard computation using the second variation of area, see \cite[Section 2]{Nunes:2011}). Thus
\begin{align*}
8(16\pi)^{-3/2} |\Sigma_t|^{3/2}
&\ge 8(16\pi)^{-3/2} |\partial_1 M|^{3/2}\\
&\ge 8(16\pi)^{-3/2} \left[\frac{4\pi}{3}(\gen-1)\right]^{3/2}\\
&\ge  \left(\frac{\gen-1}{3}\right)^{3/2},
\end{align*}
which completes the proof.
\end{proof}

\subsection{The long-time limit of inverse mean curvature flow}

First we compute the asymptotics of Ricci curvature for asymptotically locally hyperbolic manifolds. Proceeding as in Lemma 3.1 of \cite{Neves-Tian:2010}, we deduce the following:
\begin{lem}\label{RicciAsymptotics}
Let $(M^3,g)$ be a asymptotically locally hyperbolic, with radial coordinate $\rho$. If $\frac{\nabla \rho}{|\nabla \rho|}, e_1,e_2$ is an orthonormal frame at a point in $M$, then
\begin{align*}
\ric\left(\frac{\nabla \rho}{|\nabla \rho|},\frac{\nabla \rho}{|\nabla \rho|}\right)&=-2-2{\mu}\rho^{-3}+o(\rho^{-3})\\
\ric(e_i,e_j)& = -2\delta_{ij}+O(\rho^{-3})\\
\ric\left(\frac{\nabla \rho}{|\nabla \rho|},e_i\right)&=o(\rho^{-3})\\
R&=-6+o(\rho^{-3}).
\end{align*}
\end{lem}

In order to make certain computations easier, we consider a conformal compactification $\tilde{g}= \rho^{-2} g$ of the exterior region of $(M,g)$. Then if we set $s=\rho^{-1}$,  we have
\[ \tilde{g} = ds^2 + \hat{g} + \tilde{Q}, \]
 on the space $(0,s_1)\times\hat{\Sigma}$ for small enough $s_1$, where 
\[|\tilde{Q}|+ s|\tilde{\nabla}\tilde{Q}| + s^2|\tilde{\nabla}\tilde{Q}|= O_2(s^2).\]

\begin{lem}
There is a constant $C$ such that for sufficiently large $t$, the $\rho$ coordinate on $\Sigma_t$ lies in $(\frac{1}{C}e^{t/2}, Ce^{t/2})$ and the $s$ coordinate on $\Sigma_t$ must lie in $(\tfrac{1}{C}e^{-t/2}, Ce^{-t/2})$. 
\end{lem}
\begin{proof}
This follows immediately from the subsolutions and supersolutions of inverse mean curvature flow described in Lemma \ref{weak.subsolution}.  
\end{proof}

We will use the following area bound repeatedly.
\begin{lem}
Let $\tilde{\Sigma}_t$ denote the surface $\Sigma_t$ endowed with the metric induced from $\tilde{g}$. 
The area of $\tilde{\Sigma}_t$ is uniformly bounded in time.
\end{lem}
\begin{proof}
This follows immediate from the fact that $|\Sigma_t| =|\Sigma_0|\, e^t $, the definition of $\tilde{g}$, and the
previous Lemma.
 \end{proof}

\begin{lem}\label{HSquared-4}
$\int_{\Sigma_t}(H^2-4)$ is uniformly bounded above and below.
\end{lem}
\begin{proof}
The upper bound follows easily from Corollary \ref{Geroch}: Since
\[m_H(\Sigma_0)\le  m_H(\Sigma_t) = \sqrt{\frac{|\Sigma_t|}{16\pi} }\left(1-\gen-\frac{1}{16\pi}\int_{\Sigma_t}(H^2-4)\right),\]
we have
\[ \int_{\Sigma_t}(H^2-4) \le   16\pi\left(1-\gen+ m_H(\Sigma_0)\sqrt{\frac{16\pi}{|\Sigma_t|} }\right)=16\pi(1-\gen)+O(e^{-t/2}).\]

To prove the lower bound, we consider the Gauss-Codazzi equations in $(M,g)$:
\begin{equation}\label{Gauss}
2|\AA|^2+4K = H^2 +2R - 4\ric(\nu,\nu)
\end{equation}

Performing the same computation in the compactified metric $\tilde{g}$ and using the fact that integral of the left-hand side is conformally invariant, we see that
\[ \int_{\Sigma_t} (H^2 +2R - 4\ric(\nu,\nu))= \int_{\tilde{\Sigma}_t} (\tilde{H}^2 +2\tilde{R} - 4\widetilde{\ric}(\tilde{\nu},\tilde{\nu})).\]
By Lemma \ref{RicciAsymptotics}, we see that
\begin{align}
 \int_{\Sigma_t} (H^2-4) &=\int_{\tilde{\Sigma}_t} (\tilde{H}^2 +2\tilde{R} - 4\widetilde{\ric}(\tilde{\nu},\tilde{\nu})) + O(e^{-t/2}) \nonumber \\
 &\ge-C+\int_{\tilde{\Sigma}_t} \tilde{H}^2  \label{integralHsquared} 
 \ge -C,
 \end{align}
for some constant $C$ independent of $t$, where we used the fact that the curvature of $\tilde{g}$ is bounded.
\end{proof}

\begin{lem}\label{tangential-s}
We have
\[\int_{\tilde{\Sigma}_t} |\tilde{\nabla}^T s|^2=O(e^{-t/2}).\]
\end{lem}

\begin{proof}

\begin{align*}
 \int_{\tilde{\Sigma}_t} |\tilde{\nabla}^T s|^2
 &= -\int_{\tilde{\Sigma}_t} s \Delta_{\tilde{\Sigma}_t} s 
= \int_{\tilde{\Sigma}_t}s (\Delta_{\tilde{g}}s-\tilde{\nabla}_{\tilde{\nu}}\tilde{\nabla}_{\tilde{\nu}}s +\langle \tilde{H}, \partial_s\rangle)\\
&\leq \int_{\tilde{\Sigma}_t} s |\Delta_{\tilde{g}}s-\tilde{\nabla}_{\tilde{\nu}}\tilde{\nabla}_{\tilde{\nu}}s | +  \int_{\tilde{\Sigma}_t} s|\tilde{H}| \cdot|\tilde{\nabla} s| \\
&=  \int_{\tilde{\Sigma}_t} O(s^2) +  \int_{\tilde{\Sigma}_t} |\tilde{H}|\cdot O(s)\\
&\leq  
 O(e^{-t}) + O(e^{-t/2}) \left( \int_{\tilde{\Sigma}_t} \tilde{H}^2\right)^{1/2},
\end{align*}
where we used the H\"{o}lder inequality in the last line. The result now follows because inequality (\ref{integralHsquared}) states that
\[ \int_{\tilde{\Sigma}_t} \tilde{H}^2\le C+\int_{\Sigma_t} (H^2-4),\]
while Lemma \ref{HSquared-4} implies that the right-hand side is bounded. 
\end{proof}

\begin{lem}\label{AreaConvergence}
$ |\tilde{\Sigma}_{t}| =  |\hat{\Sigma}| +O(e^{-t/2}).$
\end{lem}
\begin{proof}
Choose $\tilde{\nu}$ to be the \emph{inward} pointing normal of $\tilde{\Sigma}_t$. We will first show that
\[ |\tilde{\Sigma}_{t}| \le  |\hat{\Sigma}| +O(e^{-t/2}).\]
Using the fact that $s$ is approximately a distance function with respect to $\tilde{g}$, 
\begin{multline*}
|\tilde{\Sigma}_t|  = \int_{\tilde{\Sigma}_t} (|\tilde{\nabla}s|^2 + O(s^2)) 
=\int_{\tilde{\Sigma}_t} (|\tilde{\nabla}^T s|^2+ |\tilde{\nabla}^N s|^2) + O(e^{-t}) \\
=\int_{\tilde{\Sigma}_t}  |\tilde{\nu}(s)|^2 + O(e^{-t/2}),
\end{multline*}
where the last line follows from Lemma \ref{tangential-s}. So it suffices to estimate $\int_{\tilde{\Sigma}_t}  |\tilde{\nu}(s)|^2$ in terms of $|\hat{\Sigma}|$.

We now divide $\tilde{\Sigma}_t$ into three parts: 
$$
A_t = \{x \in \tilde{\Sigma}_t\,|\,\tilde{\nu}(s) \le -e^{-t/4}\},\quad B_t = \{x \in \tilde{\Sigma}_t\,|\, -e^{-t/4}<\tilde{\nu}(s) \le 0\}$$
and
$$C_t = \{x \in \tilde{\Sigma}_t\,|\, 0<\tilde{\nu}(s) \}.$$
Then
\begin{align*}
\int_{\tilde{\Sigma}_t} |\tilde{\nu}(s)|^2
&= \int_{A_t} |\tilde{\nu}(s)|^2+\int_{B_t} |\tilde{\nu}(s)|^2+\int_{C_t} |\tilde{\nu}(s)|^2\\
&\le(1+O(e^{-t})) |A_t|  +e^{-t/2} |B_t| + (1+O(e^{-t}))\int_{C_t} \tilde{\nu}(s)\\
&\le  |A_t|  + \int_{\tilde{\Sigma}_t} \tilde{\nu}(s) + O(e^{-t/2})\\
\end{align*}
If we take $\tilde{M}_t$ to be the region of $(0,s_1)\times\hat{\Sigma}$ so that $\partial M_t=\{0\}\times \hat\Sigma\cup\Sigma_t$, then the divergence theorem tells us that
$$
\int_{\tilde{\Sigma}_t} \tilde{\nu}(s)=  \int_{\{0\}\times\hat{\Sigma}} \tilde{\nu}(s) + \int_{\tilde{M}_t} \tilde{\Delta} s
=|\hat{\Sigma}|+O(e^{-t}).
$$
Thus 
\begin{equation}\label{EstimateInTermsOfA}
\int_{\tilde{\Sigma}_t} |\tilde{\nu}(s)|^2 \le |A_t|  + |\hat{\Sigma}|+O(e^{-t/2}).
\end{equation}

It remains to estimate $|A_t|$. The mean curvature changes under under the conformal change $g=s^{-2}\tilde{g}$ according to the formula
\[ H = s\tilde{H}+2\tilde{\nu}(s). \]
Since $\Sigma_t$ evolves by inverse mean curvature flow, we know that $H\geq 0$. Therefore
\[ \tilde{\nu}(s)\geq-\frac{s\tilde{H}}{2}.\]
So the definition of $A_t$ tells us that $-e^{-t/4}\ge \tilde{\nu}(s)\geq-\frac{s\tilde{H}}{2}$ on $A_t$. Squaring this and integrating over $A_t$ gives
$$
\int_{A_t} e^{-t/2} \geq  \int_{A_t}\frac{s^2 \tilde{H}^2}{4} \implies
|A_t|  \leq  O(e^{-t/2}) \int_{\tilde{\Sigma}_t} \tilde{H}^2.
$$
As mentioned in the proof of Lemma \ref{tangential-s}, $\int_{\tilde{\Sigma}_t} \tilde{H}^2$ is bounded, 
and therefore this inequality combined with inequality (\ref{EstimateInTermsOfA}) gives us
\[\int_{\tilde{\Sigma}_t} |\tilde{\nu}(s)|^2 \le  |\hat{\Sigma}|+O(e^{-t/2}),\]
completing the proof of one side of the inequality.

The reverse inequality follows from the fact that $|\tilde{\Sigma}_t|$ is within error $O(e^{-t})$ from the area of $\tilde{\Sigma}_t$ as measured in the product metric $ds^2+\hat{g}$ on $(0,s_1)\times\hat{\Sigma}$, and the projection map of the product metric onto $\hat{\Sigma}$ is area-nonincreasing.
\end{proof}

\begin{lem}\label{tracefreeA}
There exists a sequence of times $t_i$ such that
\[\lim_{i\to\infty} \int_{\Sigma_{t_i}} |\AA|^2 = 0.\]
\end{lem}
\begin{proof}
We use formula (5.22) from \cite{Huisken-Ilmanen:2001}: For $\xi<\eta$, we have 
\begin{align*}
 \int_{\Sigma_\xi} H^2 &\geq \int_{\Sigma_\eta} H^2 + \int_\xi^\eta\int_{\Sigma_t}\left(2\frac{|DH|^2}{H^2}+2|A|^2+2\ric(\nu,\nu)-H^2\right)\,dt\\
 &\geq  \int_{\Sigma_\eta} H^2 + \int_\xi^\eta\int_{\Sigma_t}(2|\AA|^2+2\ric(\nu,\nu))
 \end{align*}
Since $\frac{d}{dt}|\Sigma_t|=|\Sigma_t|$, it follows that 
\begin{align*}
\int_{\Sigma_\xi} (H^2-4) &\geq  \int_{\Sigma_\eta} (H^2-4) + \int_\xi^\eta\int_{\Sigma_t}(2|\AA|^2+2(\ric(\nu,\nu)+2))\,dt.
\end{align*}
The last term is easily bounded:
\begin{multline*}
 \int_\xi^\eta\int_{\Sigma_t}2(\ric(\nu,\nu)+2)\,dt= \int_\xi^\eta\int_{\Sigma_t} O(\rho^{-3})\,dt
 =\int_\xi^\eta O(e^{-t/2})\,dt\\
 =O(e^{-\xi/2})
 \end{multline*}
 Hence,
 \[ \int_\xi^\eta\int_{\Sigma_t}2|\AA|^2\,dt \le \int_{\Sigma_\xi} (H^2-4) -  \int_{\Sigma_\eta} (H^2-4) +O(e^{-\xi/2}).\]
Since $ \int_{\Sigma_\xi} (H^2-4)$ is bounded above and below by Lemma \ref{HSquared-4}, the integral on the left-hand side is bounded as $\eta\to\infty$, completing the proof.
\end{proof}

\begin{lem}\label{euler.no.jump}
Using the same sequence as in Lemma \ref{tracefreeA}, we have $\chi(\Sigma_{t_i})=\chi(\hat{\Sigma})$ for sufficiently large $i$.
\end{lem}
\begin{proof}
Note that by Lemma \ref{topology}, we already know that $\chi(\Sigma_{t_i})\le\chi(\hat{\Sigma})$, so we need to show that $\chi(\Sigma_{t_i})\ge\chi(\hat{\Sigma})$ for all $i$ large.

The Gauss-Codazzi equations in the compactified metric tell us that
\[ 2|\mathring{\tilde{A}}|^2+4\tilde{K} = \tilde{H}^2 +2\tilde{R} - 4\widetilde{\ric}(\tilde{\nu},\tilde{\nu}),\]
where $\tilde \nu$ is the normal vector to the surface.
If we apply this to the surfaces $\{s\}\times \hat\Sigma$ we obtain that
\begin{equation}\label{gauss.limit}
\lim_{s\to 0}\,\left(2\tilde{R} - 4\widetilde{\ric}(\tilde\nabla s,\tilde \nabla s)\right)=4\hat k,
\end{equation}
where $\hat k$ is the constant Gaussian curvature of $(\hat \Sigma,\hat g)$.

Decompose
\[\tilde{\nu}=a\frac{\tilde{\nabla} s}{|\tilde{\nabla}s|} + v,\] 
so that $v$ is orthogonal to $\tilde{\nu}$, we have $|v|= \frac{|\tilde{\nabla}^T s|}{|\tilde{\nabla} s|}=|\tilde{\nabla}^T s|+O(s^2)$. 

Integrating Gauss-Codazzi equations on $\tilde\Sigma_t$ and recalling Lemma \ref{tangential-s} we obtain
\begin{multline*}
 \left(\int_{\tilde{\Sigma}_t} 2|\mathring{\tilde{A}}|^2\right)+8\pi\chi(\tilde{\Sigma}_t) = 
 \int_{\tilde{\Sigma}_t}  \tilde{H}^2 +2\tilde{R} - 4\widetilde{\ric}(\tilde{\nu},\tilde{\nu})
 \geq \int_{\tilde{\Sigma}_t} 2\tilde{R} - 4\widetilde{\ric}(\tilde{\nu},\tilde{\nu})\\
 = \int_{\tilde{\Sigma}_t} (2\tilde{R} - 4a^2\widetilde{\ric}(\tilde \nabla s,\tilde \nabla s)-8\widetilde{\ric}(\tilde\nabla s, v)-\widetilde{\ric}(v,v))+O(s^2)\\
 \geq  \int_{\tilde{\Sigma}_t} (2\tilde{R} - 4\widetilde{\ric}(\tilde \nabla s,\tilde \nabla s)-C(|\tilde\nabla^Ts|+|\tilde\nabla^Ts|^2))+O(s^2)\\
 = \int_{\tilde{\Sigma}_t} (2\tilde{R} - 4\widetilde{\ric}(\tilde \nabla s,\tilde \nabla s)) +O(e^{-t/4}).
\end{multline*}
Thus, using \eqref{gauss.limit}, Lemma \ref{AreaConvergence}, and  Lemma \ref{tracefreeA}, we obtain
$$\lim_{i\to\infty} 8\pi\chi(\tilde{\Sigma}_{t_i})\geq 4\hat k|\hat \Sigma|=8\pi \chi({\hat \Sigma}).$$
This implies $\chi(\tilde{\Sigma}_{t_i})\geq \chi({\hat \Sigma})$ for all $i$ sufficiently large.
\end{proof}

\begin{lem}\label{2/3-power}
If $\bar{m}\leq 0$, then 
\[ \lim_{t\to\infty} m_H(\Sigma_t) \leq 
 -\left( \fint_{\hat{\Sigma}} \mu^{2/3}\right)^{3/2}
\left( \frac{|\hat \Sigma|}{4\pi}\right)^{3/2}. \]
\end{lem}
\begin{proof}

Using  Gauss-Codazzi equations,  Gauss-Bonnet Theorem, and Lemma \ref{RicciAsymptotics}, we have 
\begin{multline*}
m_H(\Sigma_t) =\sqrt{\frac{|\Sigma_t|}{16\pi} }\left(1-\gen-\frac{1}{16\pi}\int_{\Sigma_t}(H^2-4)\right)\\
=\sqrt{\frac{|\Sigma_t|}{(16\pi)^3}}\left(16\pi(1-\gen)-\int_{\Sigma_t}(-2R +4K_{\Sigma_t} +4\ric(\nu,\nu)+|\AA|^2 -4)\right)\\
=\sqrt{\frac{|\Sigma_t|}{(16\pi)^3}}\left(8\pi\chi(\hat{\Sigma})-\int_{\Sigma_t}(4K_{\Sigma_t} +4(\ric(\nu,\nu)+2)+|\AA|^2 +o(\rho^{-3}))  \right)\\
 = \sqrt{\frac{|\Sigma_t|}{(16\pi)^3}}\left(8\pi\chi(\hat{\Sigma})-8\pi\chi(\Sigma_t)- \int_{\Sigma_t}|\AA|^2- 4\int_{\Sigma_t}( \ric(\nu,\nu)+2)+o(e^{-t/2}) \right)
\end{multline*}
If we choose $t$ to be one of the times from the sequence described in Lemma \ref{euler.no.jump}, we have
\begin{equation}\label{mass-inequality}
m_H(\Sigma_t)  \leq -\frac{1}{2}(4\pi)^{-3/2}\sqrt{|\Sigma_t|}\int_{\Sigma_t} (\ric(\nu,\nu)+2)+o(1).
\end{equation}
Decompose 
\[\nu= a \frac{\nabla \rho}{|\nabla \rho|} + v\] so that $v$ is orthogonal to $\nabla \rho$. We have $|v|= \frac{\nabla^T \rho}{|\nabla \rho|}$ while $a^2=1-|v|^2$. It follows from Lemma \ref{RicciAsymptotics} that
\begin{align}
\ric(\nu,\nu) &= a^2(-2-2\mu \rho^{-3}) + |v|^2(-2+O( \rho^{-3}))+o(\rho^{-3}) \nonumber\\
&=-2-2\mu {\rho}^{-3}+|v|^2 O(\rho^{-3}) +o(\rho^{-3}). \label{Ricci-estimate}
\end{align}
Note that 
\begin{multline*}
\int_{\Sigma_t} |v|^2 \rho^{-3}  
=\int_{\Sigma_t} \left(\frac{ |\nabla^T \rho|}{|\nabla \rho|}\right)^2 \rho^{-3}
=\int_{\Sigma_t} \left(\frac{ |\nabla^T s|}{|\nabla s|}\right)^2 s^{3}\\
=\int_{\tilde{\Sigma_t}}  \left(\frac{ |\tilde{\nabla}^T s|}{|\tilde{\nabla} s|}\right)^2 s
=O(e^{-t}),
\end{multline*}
where we used Lemma \ref{tangential-s} in the last line. Putting this together with formulas (\ref{mass-inequality}) and (\ref{Ricci-estimate}) we obtain
\begin{align}
m_H(\Sigma_t)&\leq (4\pi)^{-3/2}\sqrt{|\Sigma_t|}\int_{\Sigma_t} \mu \rho^{-3} +  o(1)\nonumber\\
&= (4\pi)^{-3/2}\left(\int_{\tilde{\Sigma}_t} s^{-2}\right)^{1/2} \left(   \int_{\tilde{\Sigma}_t} \mu s\right)+ o(1).\label{inequality1}
\end{align}
We use the H\"{o}lder inequality with $p=3$ and $q=3/2$ to see that
\begin{align*}
\int_{\tilde{\Sigma}_t} \mu^{2/3} &=\int_{\tilde{\Sigma}_t}  s^{-2/3}(\mu s)^{2/3}\\
&\leq \left(\int_{\tilde{\Sigma}_t}  s^{-2}\right)^{1/3}   \left(\int_{\tilde{\Sigma}_t} |\mu| s\right)^{2/3}.
\end{align*}
Taking the $3/2$ power of both sides, and using the fact that $\mu\le0$, we have
\[-\left(\int_{\tilde{\Sigma}_t} \mu^{2/3}\right)^{3/2} \ge  \left(\int_{\tilde{\Sigma}_t}  s^{-2}\right)^{1/2}   \left(\int_{\tilde{\Sigma}_t} \mu s\right).\]
Combining this with inequality (\ref{inequality1}) to conclude that
\[ m_H(\Sigma_t) \leq -\left(\frac{1}{4\pi} \int_{\tilde{\Sigma}_t} \mu^{2/3}\right)^{3/2} +O(e^{-t/2}),\]
for any sequence of $t$'s as in the previous lemma. Now the result follows from taking the limit as this sequence approaches infinity and applying monotonicity of $\tilde{\Sigma}_t$ and Lemma \ref{AreaConvergence}.
\end{proof}

\begin{proof}[Proof of Theorem \ref{penrose.thm}]
Recalling that $\hat g$ has constant curvature $-1$, $0$, or $1$ on~$\hat\Sigma$ and our normalizaton of area in the $\hat{k}=0$ case, we have 
\[|\hat{\Sigma}| = 4\pi\max(1,\gen-1)^{3/2}.\]
Using Geroch Monotonicity (Corollary~\ref{Geroch}) and the previous lemma, we can now conclude
\begin{multline*}
m_H(\partial_1 M)\leq \lim_{t\to\infty} m_H(\Sigma_t) \leq 
 -\left( \fint_{\tilde{\Sigma}_t} \mu^{2/3}\right)^{3/2}
\left( \frac{|\hat \Sigma|}{4\pi}\right)^{3/2}\\
\leq \bar m (\max\{1,\gen-1\})^{3/2},
\end{multline*}
where we used the definition $\bar{m}=\sup\mu$.

 All that remains is to prove the rigidity. We assume all of the hypotheses of Theorem \ref{penrose.thm}, as well as equality in the Penrose inequality. By Corollary~\ref{Geroch}, the Hawking mass $m_H(\Sigma_t)$ must be constant, and moreover, by equation (\ref{MassDifference}),
\begin{equation}\label{EqualityCase}
 \int_{\Sigma_t} \left(2(R+6)+|\AA|^2+4H^{-2}|\nabla H|^2\right)=0.
 \end{equation}
As argued on page 422 of \cite{Huisken-Ilmanen:2001}, it follows that $H$ is a positive constant on each~$\Sigma_t$, each $\Sigma_t$ is smooth, and there are no jump times. By the Smooth Start Lemma 2.4 of \cite{Huisken-Ilmanen:2001}, $\Sigma_t$ is a classical solution of inverse mean curvature flow foliating the manifold $M$. This allows us to think of inverse mean curvature flow as defining a diffeomorphism from $[0,\infty)\times \partial_1 M$ to $M$.  (In particular, $\partial M$ must be connected.) We can write the metric as
\[ g= H^{-2}dt^2 +  g_{\Sigma_t}.\]
Note that equation (\ref{EqualityCase}) implies that $R=-6$ everywhere, and $\AA=0$ on each $\Sigma_t$. In particular, $A=\frac{H}{2}g_{\Sigma_t}$ and thus $\partial_{t}g_{\Sigma_t}=\frac{2}{H}A = g_{\Sigma_t}$.
Hence
\begin{equation}\label{symmetric}
 g= H^{-2}dt^2 +  e^t g_{\Sigma_0}.
 \end{equation}
Recall that 
\[\frac{dH}{dt}= \Delta H^{-1} -(|A|^2+\ric(\nu,\nu))H^{-1},\]
which implies $\ric(\nu,\nu)$ is constant on each $\Sigma_t$ and so, by the Gauss-Codazzi equations, so is the Gauss curvature $K$. In particular, $\Sigma_0=\partial M$ has constant curvature. From here it is clear that $(M,g)$ must be the Kottler space, since Kottler spaces are  the unique  asymptotically locally hyperbolic manifolds with $R=-6$ of the form $g=f(r)^2 dr^2+r^2g_{\Sigma_0}$, where $g_{\Sigma_0}$ is a constant curvature metric on $\Sigma_0$.
\end{proof}

\begin{proof}[Proof of Theorem \ref{AHPMT}]
Assume the hypotheses of Theorem \ref{AHPMT}. If there exists at least one compact minimal surface in $M$, then the result follows from Corollary \ref{BoundaryPMT} applied to the exterior region of $M$.

So let us assume that $M$ contains no compact minimal surfaces, and suppose that $\bar{m}<0$. In particular, 
the mass aspect $\mu$ is everywhere negative. We can consider a weak inverse mean curvature flow whose initial condition is effectively a point, just as Huisken and Ilmanen did in \cite[Section 8]{Huisken-Ilmanen:2001}.  Then the same arguments used to prove Theorem \ref{penrose.thm} show that
\[0=\lim_{t\to0}m_H(\Sigma_t)\leq\lim_{t\to\infty}m_H(\Sigma_t)\leq \bar{m},\]
where the Hawking mass here has $\gen=0$. Rigidity follows according to an argument similar to the one used in Theorem \ref{penrose.thm} above.
\end{proof}

\section{The results of Chru\'{s}ciel-Simon: Proof of Theorem  \ref{CSTheorem}}\label{AreaMassInequalities}

Assume the hypotheses of Theorem \ref{CSTheorem}. In particular, we have a static data set $(M, g, V)$ and a reference Kottler space $(M_0, g_0, V_0)$ that has the same surface gravity. We also have the important assumption that $m_0\le 0$. We define 
\[W=|\nabla V|^2,\] and also a reference function $W_0$ on $M$ as follows: On the reference space~$M_0$, $|\nabla V_0|^2$ is constant on each level set of $V_0$ and, because of this, one may regard it as a function composed with $V_0$. That is, there exists a single-variable function $\omega$ with the property that $|\nabla V_0|^2 = \omega(V_0)$ as functions on $M_0$. We define the function $W_0$ on $M$ to be the function
\[ W_0 = \omega(V).\]
Recall the formula $V_0 = \sqrt{r^2+\hat{k}-\frac{2m}{r}}$ for an appropriate coordinate $r$ on~$M_0$. Inverting this formula, we may think of $r$ as some function of $V_0$. In a manner similar to the way we defined $W_0$, we can define a reference function $r$ on $M$ by composing this function with $V$. (It might be logical to call this function $r_0$, but that is unnecessary because there is no ambiguity here.)

\begin{lem}\label{elliptic}
Under the hypotheses of Theorem \ref{CSTheorem} and using the notation introduced above, if we consider an open set in $M$ on which $W$ does not vanish, then on that open set $W-W_0$ satisfies the elliptic inequality
\[\Delta(W-W_0) + \langle \xi,\nabla(W-W_0)\rangle + \alpha (W-W_0) \ge 0, 
\]
where $\xi$ is a smooth vector field and $\alpha$ is a smooth function whose sign is the opposite of $m_0$.
\end{lem}
This is the critical ingredient of the proof of Theorem \ref{CSTheorem} and it corresponds to equation (VII.15) in \cite{Chrusciel-Simon:2001}.
\begin{proof}
Since $\nabla V\ne 0$, the level sets of $V$ are smooth surfaces which we denote $\Sigma_V$. We introduce the notation
\[ \partial_V = \frac{\nabla V}{|\nabla V|^2}= \frac{\nabla V}{W}.\]
Note that $\partial_V$ has the nice property that if $f$ is a single-variable function, then 
\[\partial_V (f(V)) = f'(V).\]

Recall the static equations 
\[\Delta_g V = 3 V\quad\mbox{and}\quad\ric(g)=\frac{1}{V}\hess(V)-3 g.\]
Also note that $\nabla W = 2\hess V(\nabla V, \cdot)$ and thus
\[\partial_V W = 2\hess V\left(\frac{\nabla V}{|\nabla V|}, \frac{\nabla V}{|\nabla V|}\right)=\frac{2}{W}\hess V(\nabla V, \nabla V). \]
Choose an orthonormal frame $e_1, e_2, e_3$ such that $e_3=\frac{\nabla V}{|\nabla V|}$. Starting with the B\"{o}chner formula and the static equations above, we have
\begin{align*}
\Delta W &= 2|\hess V|^2 + 2\ric(\nabla V, \nabla V) + 2\langle \nabla(\Delta V), \nabla V\rangle \\
 &= 2\sum_{i,j=1}^3|\hess V(e_i, e_j)|^2+ \tfrac{2}{V}\hess V(\nabla V, \nabla V) \\
 &= 2|\hess V(e_3, e_3)|^2 + 4\sum_{i=1}^2|\hess V(e_3, e_i)|^2 \\
 &\quad +   2\sum_{i,j=1}^2|\hess V(e_i, e_j)|^2+ \tfrac{2}{V}\hess V(\nabla V, \nabla V)\\
 &= 2\left|\hess V\left(\tfrac{\nabla V}{|\nabla V|}, \tfrac{\nabla V}{|\nabla V|}\right)\right|^2 +\tfrac{4}{|\nabla V|^2}\sum_{i=1}^2|\hess V (\nabla V, e_i) |^2 \\
 & \quad+ 2|\nabla V|^2 |A_{\Sigma_V}|^2 + \tfrac{2|\nabla V|^2}{V}\hess V\left(\tfrac{\nabla V}{|\nabla V|}, \tfrac{\nabla V}{|\nabla V|}\right)\\
 \\
 &= \tfrac{1}{2} (\partial_V W)^2    + \tfrac{1}{W}|\nabla^T W |^2  +  2W( |\mathring{A}_{\Sigma_V}|^2+\tfrac{1}{2}H_{\Sigma_V}^2)+ \tfrac{W}{V}\partial_V W\\
 \\
 &= \tfrac{1}{2} (\partial_V W)^2+WH_{\Sigma_V}^2 + \tfrac{W}{V}\partial_V W
  + \tfrac{1}{W}|\nabla^T W |^2 +  2W |\mathring{A}_{\Sigma_V}|^2\\
  \\
 &= \tfrac{1}{2} (\partial_V W)^2+[\hess V(e_1, e_1)+\hess V(e_2, e_2)]^2+ \tfrac{W}{V}\partial_V W
  + \tfrac{1}{W}|\nabla^T W |^2 +  2W |\mathring{A}_{\Sigma_V}|^2\\
  \\
 &= \tfrac{1}{2} (\partial_V W)^2+[\Delta V - \hess V(e_3,e_3)]^2+ \tfrac{W}{V}\partial_V W
  + \tfrac{1}{W}|\nabla^T W |^2 +  2W |\mathring{A}_{\Sigma_V}|^2\\
  \\
 &= \tfrac{1}{2} (\partial_V W)^2+[3V - \hess V(e_3,e_3)]^2+ \tfrac{W}{V}\partial_V W
  + \tfrac{1}{W}|\nabla^T W |^2 +  2W |\mathring{A}_{\Sigma_V}|^2\\
  \\
 &= \tfrac{1}{2} (\partial_V W)^2+[3V - \tfrac{1}{2}\partial_V W]^2+ \tfrac{W}{V}\partial_V W  + \tfrac{1}{W}|\nabla^T W |^2 +    2W |\mathring{A}_{\Sigma_V}|^2\\
 \\
 &= \tfrac{3}{4} (\partial_V W)^2-3V\partial_V W +9V^2
+ \tfrac{W}{V}\partial_V W  + \tfrac{1}{W}|\nabla^T W |^2 +    2W |\mathring{A}_{\Sigma_V}|^2\\
\end{align*}
Meanwhile,
\begin{align*}
 \Delta W_0 &= \Div( \nabla (\omega(V)))
=\Div(\omega'(V)\nabla V )\\
 &=\omega''(V)W + \omega'(V)\Delta V
=\omega''(V)W +3 \omega'(V)V.
\end{align*}
Notice that both of the computations above provide a way to compute $\Delta_0 (\omega(V_0))$ in the reference space $M_0$. Since $\omega(V_0)$ is constant on level surfaces of $V_0$ and since those level surfaces are umbilic, the first computation yields
\begin{align*}
 \Delta_0 (\omega(V_0)) 
  &= \tfrac{3}{4} |\partial_{V_0} (\omega(V_0))|^2-3V\partial_{V_0} (\omega(V_0)) +9V_0^2
+ \tfrac{\omega(V_0)}{V_0}\partial_{V_0} (\omega(V_0)) \\
&= \tfrac{3}{4} |\omega'(V_0)|^2-3V\omega'(V_0) +9V_0^2
+ \tfrac{\omega(V_0)}{V_0}\omega'(V_0).
 \end{align*}
Meanwhile, if we do the second computation in the reference space, we find
\[ \Delta_0 (\omega(V_0))=\omega''(V_0)\omega(V_0) +3 \omega'(V_0)V_0.\]
Equating the right-hand sides of the two equations above, we actually obtain a differential equation for the single-variable function $\omega$, which means that we can replace the $V_0$ by $V$ to obtain the following equation on the original space $M$:
\begin{align*}
\omega''(V)\omega(V) +3 \omega'(V)V &= \tfrac{3}{4} |\omega'(V)|^2-3V\omega'(V) +9V^2
+ \tfrac{\omega(V)}{V}\omega'(V)\\
\omega''(V)W_0 +3 \omega'(V)V &= \tfrac{3}{4} |\partial_V W_0|^2-3V\partial_V W_0 +9V^2
+ \tfrac{W_0}{V}\partial_V W_0.
\end{align*}

Picking up from our expression for $\Delta W_0$ above, we now have
\begin{align*}
 \Delta W_0&=\omega''(V)W +3 \omega'(V)V\\
 &= \omega''(V)W_0 +3 \omega'(V)V + \omega''(V)(W-W_0) \\
&= \tfrac{3}{4} |\partial_V W_0|^2-3V\partial_V W_0 +9V^2+ \tfrac{W_0}{V}\partial_V W_0+ \omega''(V)(W-W_0).\\
\end{align*}
Now we can subtract this expression from our expression for $\Delta W$ (except for the last two nonnegative terms which we ignore) to find
\begin{multline*}
\Delta(W-W_0)\ge \tfrac{3}{4} (\partial_V W)^2-3V\partial_V W +9V^2+ \tfrac{W}{V}\partial_V W \\
\quad -\left[ \tfrac{3}{4} |\partial_V W_0|^2-3V\partial_V W_0 +9V^2+ \tfrac{W_0}{V}\partial_V W_0\right] - \omega''(V)(W-W_0)\\
= \tfrac{3}{4}\partial_V (W+W_0)\partial_V(W-W_0) -3V\partial_V (W-W_0)\\
\quad+\tfrac{W}{V}\partial_V W -  \tfrac{W}{V}\partial_V W_0 +\tfrac{W}{V}\partial_V W_0 - \tfrac{W_0}{V}\partial_V W_0- \omega''(V)(W-W_0)\\
= \left[\tfrac{3}{4}\partial_V (W+W_0)-3V+\tfrac{W}{V}\right] \partial_V(W-W_0) +(W-W_0)\tfrac{\omega'(V)}{V}- \omega''(V)(W-W_0)\\
= \left[\tfrac{3}{4}\partial_V (W+W_0)-3V+\tfrac{W}{V}\right] \partial_V(W-W_0) +[\tfrac{\omega'(V)}{V}- \omega''(V)](W-W_0).
\end{multline*}
It only remains to compute the sign of the zero order coefficient. Using the the formula
\[V_0 = \left(r^2+\hat{k}-\frac{2m_0}{r} \right)^{1/2}\]
in $M_0$, we can perform simple computations using the chain rule to show that
\[ \omega(V) = \left(r+\frac{m_0}{r^2}\right)^2. \]
and 
\[ \frac{dV}{dr} = \frac{\sqrt{\omega(V)}}{V}. \]
Therefore
\[ \omega'(V) = 2V\left(1-\frac{2m_0}{r^3}\right),\]
and
\begin{align*} 
\omega''(V) &=2V\left(1-\frac{2m_0}{r^3}\right)  + 2V\frac{6m_0}{r^4}\frac{dr}{dV}\\
 &=\frac{\omega'(V)}{V}+ \frac{12V^2}{\sqrt{\omega(V)}}\frac{m_0}{r^4}.
 \end{align*}
 Thus
 \[\tfrac{\omega'(V)}{V}- \omega''(V) = - \frac{12V^2}{\sqrt{W_0}}\frac{m_0}{r^4}.\]
So we see that as long as $m_0\le 0$, the coefficient of $W-W_0$ is nonnegative.

\end{proof}

\begin{cor}\label{compare}
Under the hypotheses of Theorem \ref{CSTheorem} and using the notation introduced above, we have
$W\le W_0$ on $M$. 
\end{cor}
\begin{proof}
By the definition of $W_0$, we know that $W=W_0$ on $\partial_1 M$, and $W\le W_0$ on all of $\partial M$ since $\partial_1 M$ was assumed to have the largest surface gravity of any boundary component. We also know from Proposition \ref{StaticAsymptotics} that \mbox{$W-W_0=0$} at the conformal infinity. (We will see a more detailed calculation below.) Suppose that $W>W_0$ somewhere in $M$. Since $W\le W_0$ at the boundary and also ``at infinity,'' the function $W-W_0$ must achieve its positive maximum value at some point $p$ in the interior of $M$. Since $W(p)>W_0(p)>0$, we see that the previous lemma applies to some open set $U$ containing $p$. By the maximum principle for the elliptic inequality given by the Lemma, $W-W_0$ cannot have a local maximum in $U$, which is a contradiction.
\end{proof}

\begin{proof}[Proof of Theorem \ref{CSTheorem}]
We claim that for any point $p\in\partial_1 M$, 
\[ K_{\partial_1 M}(p) \ge K_{\partial M_0},\]
where $K_{\partial_1 M}(p)$ is the Gaussian curvature of $\partial_1 M$ at $p$, while  $K_{\partial M_0}$ is the constant Gaussian curvature of the $\partial M_0$ in the reference space. Integrating this inequality and using the Gauss-Bonnet Theorem, the claim clearly implies that
\[ \frac{\chi(\partial_1 M)}{|\partial_1 M|} \ge \frac{\chi(\partial M_0)}{|\partial_1 M_0| }, \]
so we now focus on proving the claim.

For the following, we use the same notation as in the proof of Lemma~\ref{elliptic}.  Since $W\ne 0$ at $\partial M$, the vector field $\partial_V$ is well-defined near $\partial_1 M$, and we can consider the flow $\varphi_v$ generated by $\partial_V$ near $\partial_1 M$. Choose a point $p\in\partial_1 M$. Applying Taylor's Theorem to the function $W$ restricted to the flow line starting at $p$, we have
\[ W(\varphi_v(p))= W(p) + [(\partial_V W)(p)] v +  [(\partial^2_V W)(p)] v^2 + O(v^3).\] 
We know that $W(p)=\kappa^2$, where $\kappa$ is the surface gravity of $\partial_1 M$. Next,
\begin{align*}
\partial_V W&= 2\hess V(\nabla V, \partial_V)
 = 2\hess V(e_3, e_3)\\
&= V(2\ric(e_3, e_3) +6)=0,
\end{align*}
where the last identity follows because $V$ vanishes on $\partial M$.  Following same reason and using the Gauss-Codazzi equations plus the fact that $\partial M$ is totally geodesic, we have 
\begin{multline*}
\partial^2_V W= (2\ric(e_3, e_3) + 6) + V \partial_V(2\ric(e_3, e_3) + 6)\\
=2\ric_p(e_3, e_3) + 6=R-2K_{\partial_1 M}(p)+6=-2K_{\partial_1 M}.
\end{multline*}
Thus
\[ W(\varphi_v(p))= \kappa^2 -2[K_{\partial_1 M}(p)]v^2+O(v^3).\]
Performing the same computation in the reference solution, we obtain
\[ W_0(\varphi_v(p))= \kappa^2 -2K_{\partial M_0}v^2+O(v^3).\]
The claim now follows from Corollary \ref{compare}.

We now focus on proving the mass inequality $\mu\le m_0$. Recall from Proposition \ref{StaticAsymptotics} that we have
\[V^2 = \rho^2+\hat{k} -\tfrac{4}{3}\mu\rho^{-1}+o_1(\rho^{-1}),\]
where $\rho$ is a coordinate as in the definition of asymptotically locally hyperbolic.
Differentiating this, we find
\[2V\nabla V = (2\rho+\tfrac{4}{3}\mu\rho^{-2}) \nabla \rho+o(\rho^{-1}).\]
Taking the norm-square of both sides, and using the asymptotically locally hyperbolic property,
\begin{align*}
V^2 W &= (\rho+\tfrac{2}{3}\mu\rho^{-2})^2|\nabla\rho|^2+o(\rho) \\
&= (\rho^2+\tfrac{4}{3}\mu\rho^{-1})(\rho^2+\hat{k})+o(\rho) \\
&= \rho^4+\hat{k}\rho^2+\tfrac{4}{3}\mu\rho+o(\rho).
\end{align*}
Then
\begin{align*}
W &= \frac{\rho^4+\hat{k}\rho^2+\tfrac{4}{3}\mu\rho+o(\rho)}{\rho^2+\hat{k} -\tfrac{4}{3}\mu\rho^{-1}+o(\rho^{-1})}\\
 &=\rho^2 \frac{ 1+\hat{k}\rho^{-2}+\tfrac{4}{3}\mu\rho^{-3}}{1+\hat{k}\rho^{-2} - \tfrac{4}{3}\mu\rho^{-3}}+o(\rho^{-1})\\
 &= \rho^2+\tfrac{8}{3}\mu\rho^{-1}+o(\rho^{-1}).
\end{align*}
Recall that we also have 
\[r^2+\hat{k}-\frac{2m_0}{r}=V^2= \rho^2+\hat{k} -\tfrac{4}{3}\mu\rho^{-1}+o_1(\rho^{-1}),\]
 by definition of the function $r$. In particular, 
 \[ r^{-1}= \rho^{-1}+o(\rho^{-1}).\]
 By changing variables from $r$ to $\rho$, we have
\begin{align*}
W_0 &= \left(r+\frac{m_0}{r^2}\right)^2
 = r^2 + 2m_0 r^{-1} +O(r^{-4})\\
 &= (r^2+\hat{k}-2m_0 r^{-1})-\hat{k} + 4m_0 r^{-1} +O(r^{-4})\\
 &= \rho^2 -\tfrac{4}{3}\mu\rho^{-1}+ 4m_0 \rho^{-1} +o(\rho^{-1}).
\end{align*}
Comparing these asymptotic expansions for $W$ and $W_0$ and using Corollary \ref{compare}, we see that as $\rho\to\infty$, we have
\[ \tfrac{8}{3}\mu \le -\tfrac{4}{3}\mu+4m_0,\]
from which the result follows.
\end{proof}

\bibliographystyle{hplain}
\bibliography{2012references}

\end{document}